\documentclass[11pt]{article}
\usepackage{amsmath,amsfonts,amssymb}
\usepackage{graphicx, color}

\def\forall{\hbox{for all}~}

\def\ve{\varepsilon}

\def\D{{\cal D}}

\def\caL{{\cal L}}
\def\R{{\mathbb R}}%{I\!\!R}

\def\implies{\Longrightarrow}

\def\vs{\vskip 2em}

\def\v{\vskip 1em}

\def\H{{\cal H}}

\def\ov{\overline}
\def\Tilde{\widetilde}

\def\bd{\partial^*}
\def\bega{\begin{array}}
\def\enda{\end{array}}
\def\begi{\begin{itemize}}
\def\endi{\end{itemize}}

\def\ds{\displaystyle}
\def\bel{\begin{equation}\label}
\def\eeq{\end{equation}}
\def\sqr#1#2{\vbox{\hrule height .#2pt
\hbox{\vrule width .#2pt height #1pt \kern #1pt
\vrule width .#2pt}\hrule height .#2pt }}
\def\square{\sqr74}
\def\endproof{\hphantom{MM}\hfill\llap{$\square$}\goodbreak}

\usepackage{url}

\newtheorem{theorem}{Theorem}[section]

\newtheorem{lemma}{Lemma}[section]

\newtheorem{proposition}{Proposition}[section]
\newtheorem{remark}{Remark}[section]

\begin{document}
\title{\bf On the Minimum Speed for a Fire Blocking Problem}
\vs
\author{Alberto Bressan$^{(1)}$ and Luca Talamini$^{(2)}$
\\
\, \\
{\small $^{(1)}$Department of Mathematics, Penn State University,} {\small University Park, PA~16802, USA.}\\
 {\small $^{(2)}$~SISSA, Via Bonomea 265, Trieste 34136,  Italy.}\\
\, \\
{\small E-mails: axb62@psu.edu,~ltalamin@sissa.it}
}
\maketitle

\begin{abstract} 
We consider a dynamic blocking problem: a
fire, starting on a bounded open set $R_0\subset\R^2$, propagates  in the plane with unit speed in all directions.
To block it, a barrier can be constructed in real time,
at speed $\lambda>0$. A long standing conjecture by the first author states that the fire can be entirely
surrounded by the barrier, in finite time,
if and only if $\lambda>2$.   The present paper improves on earlier lower bounds. Namely,  if $\lambda\leq \Lambda=1.62621170228\ldots $
we show that the fire cannot be blocked.     We also prove that, if $\lambda\leq 2$, the fire cannot be blocked by 
constructing a barrier which is a simple closed curve.   More precisely, if the barrier does not contain purely delaying arcs, i.e.~arcs which are eventually surrounded by the fire on both sides, then for  $\lambda<2$ the
region burned by the fire at time $t>0$ has area $A(t)>(2-\lambda )t^2/4$.
\end{abstract}

{ \small {\bf Keywords}: dynamic blocking problem, Hamilton-Jacobi equation with obstacles,
isoperimetric inequality, wildfire confinement.\\
{\bf MSC}: 34A60, 49Q20.
}

\section{Introduction}
\label{sec:1}
\setcounter{equation}{0}

\v
Aim of this note is to analyze the dynamic blocking problem introduced
in \cite{B}, originally
motivated by the control of wild fires or of the
spatial spreading of a contaminating agent.

At each time $t\geq 0$, we denote by $R_t\subset\R^2$
the {\em burned region}.
In the absence of control,
we assume that the set $R(t)$ grows uniformly in all directions, namely
$$R_t\doteq B(R_0, t)~\doteq \Big\{ x\in\R^2\,;~~ d(x, R_0)\leq t\Big\}.
$$
Here $R_0\subset\R^2$ is a fixed  (nonempty) bounded open set, describing the
region burned by the fire at the initial time $t=0$.
In our model,  the spreading of the fire can be
controlled by constructing barriers. In mathematical terms,
we thus assume that the
controller can construct a one-dimensional rectifiable curve
$\gamma$ which blocks the spreading of the fire.
Calling $\gamma(t)\subset\R^2$ the portion of the wall constructed within
time $t\geq 0$, and denoting by $\H^1$  the one-dimensional Hausdorff measure, we assume:
\begi
\item[{\bf (A1)}]~For every $t_2>t_1\geq 0$ one has $\gamma(t_1)\subseteq\gamma(t_2)\,$.
\item[{\bf (A2)}]~For every $t\geq 0$, the total length of the barrier satisfies
\bel{1.4}\H^1(\gamma(t)) \leq \lambda t.\eeq
\endi
By {\bf (A1)}, after a portion of the barrier has been constructed, it remains in place forever and 
cannot be moved to any other location.  
The constant $\lambda>0$ in (\ref{1.4}) 
is the speed at which the wall
is constructed.
A barrier $\gamma(\cdot)$ satisfying {\bf (A1)-(A2)} will be
called a {\em $\lambda$-admissible strategy}.  
%
% In addition, we say that
%the strategy $\gamma$ is {\em complete} if it satisfies
%\v
%{(H3)} For every $t\geq 0$ there holds
%\bel{1.5} \H^1(\gamma(t)) = \lambda t\,,
%\qquad\qquad \gamma(t)=\bigcap_{s>t}\gamma(s)\,.\eeq
%Moreover, if $\gamma(t)$ has positive upper
%density at
%a point $x$, i.e.~if
%$$\limsup_{r\to 0+}{ \H^1\Big( B(x,r)\cap\gamma(t)\Big)\over r}>0\,,$$
%then $x\in\gamma(t)$.  Here $B(x,r)$ is the ball centered at $x$
%with radius $r$.
%\v
%As proved in \cite{BDL}, for every admissible strategy $t\mapsto \gamma(t)$
%one can construct a second admissible strategy $t\mapsto \tilde\gamma(t)
%\supseteq\gamma(t)$,
%which is complete.
By constructing barriers, the burned set is reduced.
Indeed, this set is 
\bel{1.6}\begin{array}{rl}R^\gamma_t~\doteq&\Big\{ x(t)\,;~~x(\cdot)~~
\hbox{is absolutely continuous}\,,~~
x(0)\in R_0\,,\cr
&\qquad\quad |\dot x(\tau)|\leq 1 ~~\hbox{for a.e.}~
\tau\in [0,t]\,,\qquad x(\tau)\notin \gamma(\tau)~~
\hbox{for all}~\tau\in [0,t]\Big\}.
\end{array}\eeq
In the above setting, following \cite{B} we consider
\v
\begi
\item[{\bf(BP)}] {\bf Blocking Problem.}  {\it Given an initial bounded open set $R_0\subset \R^2$,
find an admissible strategy
$t\mapsto\gamma(t)$ such that the corresponding reachable sets
$R^\gamma_t$ remain uniformly bounded.}
\endi
In other words,
we seek a strategy such that
 $$R^\gamma_t~\subseteq~ B_R~\doteq~\{x\in\R^2\,;~|x|< R\} \qquad\forall t\geq 0,$$
for some radius $R$ sufficiently large.
A key question that still  remains open is
\begi
\item[{\bf (Q)}] {\it For which construction speeds $\lambda$ can the fire be blocked ?}
\endi
By  rescaling argument, it is clear that the answer does not depend on the initial set $R_0$.  Indeed,
 if $\gamma(\cdot)$ is a $\lambda$-admissible strategy that blocks a fire starting from the set $R_0$, then for every $c>0$ the rescaled strategy
$$\Tilde\gamma(t)~=~c \,\gamma(t/c)$$
is a $\lambda$-admissible strategy that 
blocks a fire initially burning on the set $c R_0$.
W.l.o.g.~one can thus assume 
that $R_0=\bigl\{x\in\R^2\,;~|x|<1\bigr\}$ is the unit disc.

A long standing conjecture by the first author is that the fire can be blocked only if $\lambda>2$.
In this direction, the following results are known.
\begi
\item[(1)] When $\lambda>2$ the fire can be blocked.  When $\lambda\leq 1$ it cannot be blocked \cite{B}.
\item[(2)] If the fire propagation is restricted to a half plane, then it can be blocked if and only if $\lambda >1$, see \cite{BW2}.
\item[(3)] 
%If the entire barrier is a simple closed curve, then the blocking strategy can be admissible only if $\lambda >2$.
A spiral-like strategy which blocks the fire can be constructed only if $\lambda>\lambda^*\doteq  2.614430844\ldots$
\cite{BZ, BBFJ}.
\item[(4)] If $\lambda\leq 3/2$, the fire cannot be blocked \cite{Schmidt}.
\endi
%It has remained an open question whether a blocking strategy exists when $\lambda \in \, ]3/2, 2]$.
%By (3), if such a strategy exists, the barrier cannot be a simple closed curve.  Namely, it must also contain
%{\bf delaying arcs}, which are eventually surrounded by the fire on both sides.  
%Such arcs have the only purpose to slow down 
%the fire advancement.
In the present note we point out a very simple argument showing that, when $\lambda\leq \pi/2$, 
the fire cannot be blocked.  Indeed, calling $A(t)$ the area of the burned region at time $t\geq 0$,
by the isoperimetric inequality the boundary of the burned set has length $\geq 2\sqrt{\pi A(t)}$.
Of this boundary, at most a portion of length $\lambda t$ can be covered by walls. 
On the remaining portion $F_t$ the 
fire expands freely (see Fig.~\ref{f:fc207}, left).   This yields the lower bound
\bel{dat}{d\over dt} A(t)~\geq~ 2\sqrt{\pi A(t)}-\lambda t.\eeq
When $\lambda\leq \pi/2$, this differential inequality yields $A(t)\geq c t^2$ for some $c>0$ and all $t\geq 0$, showing that the fire cannot be blocked.

This lower bound on the blocking speed $\lambda$ can be improved by observing that,
 for any $\tau>t$, the 
barrier $\gamma(\tau)$ must contain $\gamma(t)$ (see Fig.~\ref{f:fc207}, right).  Hence at one of the two times  $t$ or $\tau$ the 
inequality (\ref{dat}) must be strict.   A lower bound on the sum $A'(t) + A'(\tau)$ can now be obtained
by looking at a double-bubble minimization problem,  for the weighted sum of the perimeters of two
sets with areas $A(t)$ and $A(\tau)-A(t)$.

Given a time $T>0$, for any couple of times $t\in [0, T/2]$ and $\tau = t+T/2$ the bounds  
known for the double-bubble problem \cite{ABFHZ, Lawlor, R} provide a sharper estimate on the lengths
of the free boundaries $F_t$ and $F_\tau$ where the fire expands.
In particular, when $\lambda\leq\Lambda\doteq 1.62621170228\ldots$
an integration yields
$$A(T)~=~A(0)+\int_0^{T/2} \bigl[ A'(t) + A'(t+T/2)\bigr]\, dt~\geq ~c_* T^2$$
for some $c_*>0$ and all $T\geq 0$.  Again, this shows that the fire cannot be blocked.

The last section of the paper is concerned with  the case where the entire barrier $\Gamma=\cup_{t>0} \gamma(t)$
is a Jordan curve, i.e.~a simple closed curve.   In such case we show that the fire can be blocked 
if and only if $\lambda>2$.    If $\lambda<2$ and the barrier does not contain purely delaying arcs, i.e.~arcs that
are eventually surrounded by the fire from both sides, 
we prove that the area of the burned set satisfies $A(t)\geq (2-\lambda)t^2/4$.
This indicates that, if a $\lambda$-admissible blocking strategy were to exist for some $\lambda<2$,
it would have to contain some barriers whose only purpose is to delay the spreading of the fire.
As time increases, these barriers eventually lie in the interior of the closure of the burned set.

%
%We recall that the Hausdorff distance between two compact sets
%$X,Y$
%is defined as
%$$d_H(X,Y)\doteq \max\Big\{ \max_{x\in X} ~d(x,Y)~,~
%\max_{y\in Y} d(y,X)\Big\}\,,$$
%where
%$d(x,Y)\doteq \inf_{y\in Y} d(x,y)$ and
%$d(x,y)\doteq |x-y|$ is the Euclidean distance on $\R^2$.
%
For the definition and properties of rectifiable sets we refer to \cite{AFP, F, M}.
Related problems of fire confinement have been studied in \cite{BW, KLL}.
In connection with a cost functional accounting for the value of the burned region plus the cost of constructing the barrier, the existence of optimal barriers was proved in  \cite{BDL, DLR}.   Necessary conditions for optimality,
describing each arc of the optimal barrier as the solution to a suitable ODE, were derived in \cite{B, BW3, WT}.

\v

\begin{figure}[ht]
\centerline{\hbox{\includegraphics[width=10cm]{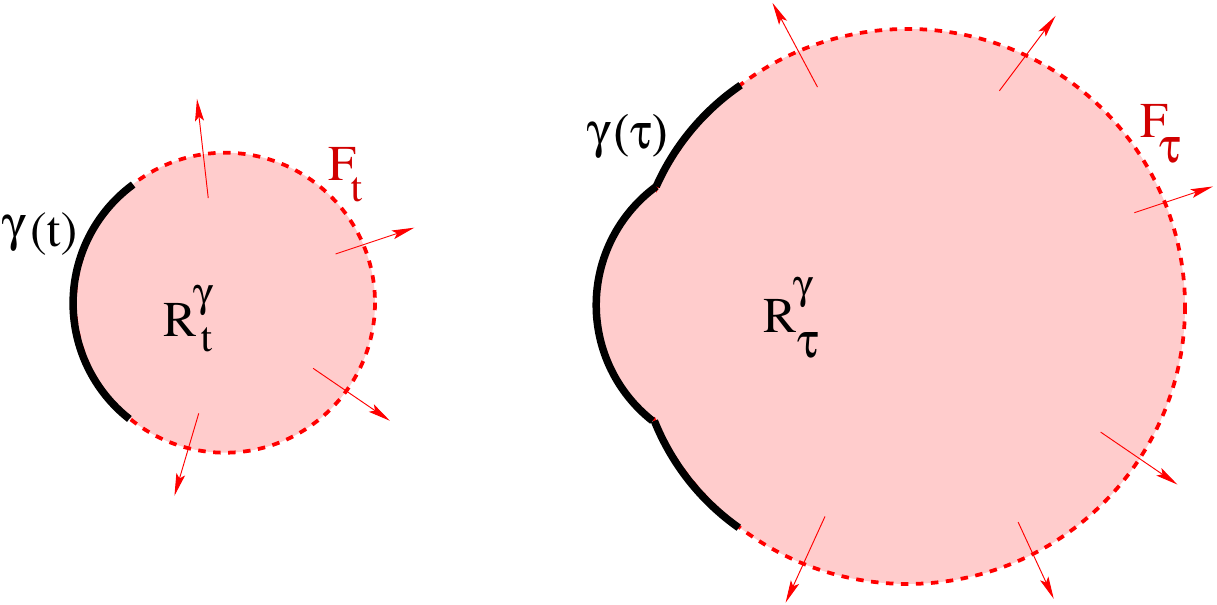}}}
\caption{\small Left: by the isoperimetric inequality, at any time $t>0$ the portion $F_t$ 
of the boundary where the fire expands has length $\geq 2\sqrt{\pi A(t)} -\lambda t $.   
Right: since for $\tau>t$ the 
barrier $\gamma(\tau)$ must contain $\gamma(t)$, at one of the two times  $t$ or $\tau$ the 
inequality (\ref{dat}) must be strict.   For this reason, the bound in (\ref{Atbo}) can be slightly improved.}
\label{f:fc207}
\end{figure}
\section{Preliminary results on the minimum time function}
\label{sec:2}
\setcounter{equation}{0}

Throughout this section we assume that $R_0$ is a bounded  open set, whose boundary has finite length:
$\H^1(\partial R_0)<+\infty$.

Consider the minimum time function associated with an admissible
strategy $\gamma$:
\bel{time}
u^\gamma(x)~\doteq~
\inf\bigl\{t>0\,;~~x\in R^\gamma_t\bigr\},
\eeq
with the convention $\inf\emptyset=+\infty$.

Following \cite{AFP, F}, by a {\em rectifiable set}  we mean a
set $\Gamma\subset
\R^2$ which can be decomposed as
$\Gamma=\Gamma_0\cup\Gamma_1$,
where ${\H^1}(\Gamma_0)=0$,
${\H^1}(\Gamma_1)<\infty$, and $\Gamma_1$ is contained in a countable union of
Lipschitz continuous curves.
We say that the rectifiable set $\Gamma$ is {\em complete} if it contains all
of its points
of positive upper density:
$$%\theta^*(\Gamma;x)~\doteq~
\limsup_{r\to 0+}{ {\H^1}\Big( B(x,r)\cap\Gamma\Big)\over r}~>~0
\qquad\implies\qquad x\in\Gamma
\,.$$
We recall that, if $\Sigma$ is any rectifiable set (possibly not complete),
the closure $\overline \Sigma$ may well be the entire plane $\R^2$.
However, the union
$$\Sigma^*~\doteq~\Sigma~\cup~\left\{ x\in\R^2\,;~~
\limsup_{r\to 0+}{ {\H^1}\Big( B(x,r)\cap\Sigma\Big)\over r}>0
\right\}$$
is a complete rectifiable set, with ${\H^1}(\Sigma^*) = {\H^1}(\Sigma)$.
Throughout the following, we refer to $\Sigma^*$ as the {\em completion}
of the set $\Sigma$.

Following \cite{BDL, DLR}, an admissible strategy is called
{\em complete} if, for every $t\geq0$,
\[
\gamma(t)=\bigcap_{s>t}\gamma(s),
\]
and moreover $\gamma(t)$ contains every point $x\in\mathbb R^2$ satisfying
\[
\limsup_{r\downarrow0}
\frac{\H^1\bigl(\gamma(t)\cap B_r(x)\bigr)}{r}\,>\,0.
\]
We recall that by \cite[Theorem 1]{DLR} the minimum-time function $v=u^\gamma:\mathbb R^2\to[0,+\infty]$ corresponding to an admissible strategy $\gamma(\cdot)$ satisfies the following regularity properties.
For every $s>0$, the truncation $v_s=\min\{v,s\}$ satisfies
\[
v_s\in SBV_{\rm loc}(\mathbb R^2),\qquad
v_s=0\quad\hbox{a.e. on }R_0,
\]
and moreover the jump set of $v_s$ and the absolutely continuous part of its gradient satisfy
\[
\H^1(J_{v_s}\setminus\gamma(s))=0,
\qquad
|\nabla^a v_s|\leq1\quad\hbox{a.e.}
\]

We next relate the growth of the burned area to the free front. Here and in the sequel, $|\Omega|$
denotes the Lebesgue measure of a set $\Omega\subset\R^2$, while $\partial^*\Omega$ denotes its reduced boundary (see \cite{AFP, M} for precise definitions).

\begin{proposition}\label{prop:arearate}
Let $\gamma$ be a complete, admissible strategy and set
$A(t)=|R^\gamma_t|$, with $A(0)=|R_0|$.
Then $A$ is nondecreasing and $R^\gamma_t$ has finite perimeter
for almost every $t>0$.
For such times, define
\bel{WFt}
W_t\,\doteq\,\gamma(t)\cap\partial^\ast R^\gamma_t,
\qquad
F_t\,\doteq\,\partial^\ast R^\gamma_t\setminus\gamma(t).
\eeq
Then
\[
\partial^\ast R^\gamma_t\,=\,W_t\mathbin{\cup}F_t,\qquad\qquad W_t\mathbin{\cap}F_t\,=\,\emptyset.
\]
Moreover, for all $\tau_1 < \tau_2$ there holds
\begin{equation}\label{eq:areaineq}
A(\tau_2)~\geq~ A(\tau_1) +\int_{\tau_1}^{\tau_2}\H^1(F_s)\,ds.
\end{equation}
\end{proposition}

{\bf Proof.} Fix $\tau > \tau_2$. By the coarea formula applied for the absolutely continuous part of the gradient of $u^\gamma_\tau$, we have
\[
  A(\tau_2)- A(\tau_1) ~\geq~  |\nabla^a u^\gamma_\tau|\bigl(\{\tau_1 < u^\gamma_\tau \leq \tau_2\}\bigr) ~ = ~\int_{\tau_1}^{\tau_2} \H^1 \bigl(\partial^\ast \{u^\gamma_\tau \leq s\} \setminus J_{u^\gamma_\tau} \bigr)\, d s.
\]
Next, we claim that, for almost every $s<\tau$, up to $\H^1$-negligible sets,
\begin{equation}\label{eq:twoincl}
J_{u^\gamma_\tau}\cap
\partial^\ast\{u^\gamma_\tau\leq s\}
\,\subseteq \,J_{u^\gamma_s}\,
\subseteq\,\gamma(s)
\end{equation}
 Indeed, for almost every time $s$, we claim that the following holds. For $\mathcal{H}^1$-a.e. $x \in J_{u^\gamma_\tau}\cap
\partial^\ast\{u^\gamma_\tau\leq s\}$, the left and right traces $(u_\tau^\gamma)^{\pm}(x)$ satisfy 
$$
(u^\gamma_\tau)^-(x)  < s < (u^\gamma_\tau)^+(x).
$$
This is obtained using Fubini's theorem:
\begin{equation}
    \begin{aligned}
        0 & = \int_{  J_{u_\tau^\gamma}   } \int_{\mathbb R} \mathbf 1_{\{u^{\pm}(x) = s\}}(x) \, d  s \; d  \H^1(x)\\
        & = \int_{\mathbb R} \int_{  J_{u_\tau^\gamma}   } \mathbf 1_{\{u^{\pm}(x) = s\}}(x) \;  d  \H^1(x) \, d  s\\
        &  = \int_{\mathbb R} \H^1\Big(\big\{x \in  J_{u_\tau^\gamma} \; | \; u^{\pm}(x) = s   \big\} \Big)ds .
    \end{aligned}
\end{equation}
Therefore, for $s$ outside a negligible set,
\begin{equation}\label{eq:fubin}
    (u_\tau^\gamma)^{\pm}(x) \neq s \qquad \text{for $\H^1$-a.e. point $x \in  J_{u_\tau^\gamma}$}.
\end{equation}
If a point $x$ additionally belongs to $ J_{u_\tau^\gamma} \cap \partial^\ast \{u_\tau^\gamma \leq s\}$, then this means that the upper trace is greater than $s$ and the lower trace is less than $s$,  i.e. $(u_\tau^\gamma)^+(x) \geq s \geq (u_\tau^\gamma)^-(x)$, for $\H^1$-a.e. point in this intersection. Together with \eqref{eq:fubin}  this 
proves the claim.  
% {\color{red} $x$ is a jump point of $u^\gamma_\tau$ for $\H^1$ a.e. $x \in J_{u^\gamma_\tau}\cap
% \partial^\ast\{u^\gamma_\tau\leq s\}$, with traces that satisfy $(u^\gamma_\tau)^-(x)  < s < (u^\gamma_\tau)^+(x)$. }
Therefore, for such $x$, we immediately obtain $x \in J_{u_s^\gamma}$. The second inclusion in \eqref{eq:twoincl} follows by the assumption that $\gamma(\cdot)$ is complete.

Consequently,
\[
A(\tau_2)-A(\tau_1)
~\geq~\int_{\tau_1}^{\tau_2}
\H^1\bigl(
\partial^\ast\{u^\gamma_\tau\leq s\}
\setminus\gamma(s)\bigr)\,ds\\
~=~\int_{\tau_1}^{\tau_2}\H^1(F_s)\,ds.
\]
\endproof
\begin{remark} {\rm A more careful proof would show 
 that \eqref{eq:areaineq} holds with equality. However, this is not needed for our present purposes. }
\end{remark}

\section{An elementary lower bound on the blocking speed}
\label{sec:3}
\setcounter{equation}{0}

In this section we provide a first, elementary lower bound for the minimum speed blocking the fire. We recall that $A(t) = |R^\gamma_t|$ is the area of the burned region at time $t > 0$, given an admissible strategy $\gamma$. We assume here that the initial set $R_0 = B_1$ is the unit ball, so that $A(0) = \pi$.

\begin{theorem}  Let $\gamma$ be any admissible strategy. In the above setting, 
if
\(0\leq\lambda\leq\pi/2\), then
\[
A(t)~\geq~ \pi(1+a_\lambda t)^2,
\qquad\hbox{where}\qquad
 a_\lambda\,\doteq\,\frac{1+\sqrt{1-2\lambda/\pi}}{2}\,.
\]
In particular, when \(\lambda=\pi/2\) one has
\bel{Atbo}
A(t)~\geq ~\pi\Bigl(1+\frac t2\Bigr)^2,
\eeq
hence the fire cannot be contained.  
\end{theorem}

{\bf Proof.}
By Proposition \ref{prop:arearate}, there holds
\bel{A'} A'(t)\,\geq\,\H^1(F_t)
\eeq
for a.e. $t > 0$.
On the other hand, the barrier constructed up to time $t$ has length at most 
$\H^1\bigl(\gamma(t)\bigr)\leq \lambda t$.  
Therefore
\[
\H^1\bigl(\partial^* R^\gamma_t\bigr)= \H^1(F_t) + \H^1(W_t)~\leq ~ A'(t)+\lambda t.
\]
The planar isoperimetric inequality yields
\begin{equation}\label{2.1}
A'(t)~\geq ~\H^1\bigl(\partial^* R^\gamma_t\bigr)-\lambda t
      ~ \geq~ 2\sqrt{\pi A(t)}-\lambda t.
\end{equation}

Now set
\[
Q(t)\,\doteq~\pi(1+at)^2
\]
and choose \(a\in[1/2,1]\) so that
\begin{equation}\label{eq:a}
2\pi a(1-a)=\lambda.
\end{equation}
Such a number \(a\) exists precisely when \(\lambda\leq\pi/2\). Indeed
\[
\max_{0\leq a\leq1}2\pi a(1-a)=\frac\pi2.
\]
Moreover, by (\ref{eq:a})
\begin{align*}
2\sqrt{\pi Q(t)}-\lambda t-Q'(t)
&~=~2\pi(1-a)(1+at)-\lambda t\\
&~=~2\pi(1-a)~\geq~0,
\end{align*}
Therefore
\[
Q'(t)~\leq~2\sqrt{\pi Q(t)}-\lambda t.
\]

We now observe \(Q(0)=A(0)=\pi\).   Moreover, at any time $t$ where $A(t)= Q(t)$, by (\ref{2.1}) one has
 \bel{comp1}  A'(t)~\geq~2\sqrt{\pi Q(t)}-\lambda t~\geq ~Q'(t)\,.\eeq
A standard comparison argument now yields
\[
A(t)\,\geq\, Q(t)\qquad\forall t\geq0.
\]
   
At the critical value \(\lambda=\pi/2\)  the equation \eqref{eq:a} has the
double root \(a=1/2\).  Hence
\[
A(t)~\geq~\pi\Bigl(1+\frac t2\Bigr)^2~\to~\infty,
\]
showing that the fire cannot be blocked.  Of course, the same conclusion holds for every
\(\lambda<\pi/2\).
\endproof

\section{Improving the lower bound}
\label{sec:4}
\setcounter{equation}{0}

The earlier inequality (\ref{2.1}) did not use the fact that for  different times $t$ the sets $\gamma(t)$
belong to the same growing
barrier.  Relying on that additional information, the estimate can be somewhat improved.

Recall the definition of the {\em active wall} at time $t$, that is
$$W_t~\doteq~\gamma(t)\cap \bd R^\gamma_t\,.$$
Notice that the remaining portion of the barrier $\gamma(t)\setminus W_t$ consists of points where
(see Fig.~\ref{f:fc208})
\begi 
\item either the barrier has not yet been touched by the fire,
\item or the fire has already surrounded the barrier on both sides.
\endi

Fix two times where Proposition \ref{prop:arearate} applies, and consider the set
\[
\bd R^\gamma_a\cap\bd R^\gamma_b~=~ W_a\cap W_b.
\]
Note that these are the points which lie on an active wall at time $a$, which remains active until time $b$.

Introduce the function
\begin{equation}\label{eq:Qdef}
 Q(a,b)~\doteq~\H^1(\partial^* R^\gamma_a)+\H^1(\partial^* R^\gamma_b)-(2-\sqrt2 )\H^1(\bd R^\gamma_a\cap\bd R^\gamma_b).
\end{equation}
The sets
\[
 R^\gamma_a,
 \qquad
 R^\gamma_b\setminus R^\gamma_a,
 \qquad
 \R^2\setminus R^\gamma_b
\]
form a  partition of $\R^2$.  In $Q(a,b)$ the interfaces between consecutive
phases have weight $1$, while the interface separating $R^\gamma_a$ directly from the
exterior has weight $\sqrt2$.  

\begin{figure}[ht]
\centerline{\hbox{\includegraphics[width=12cm]{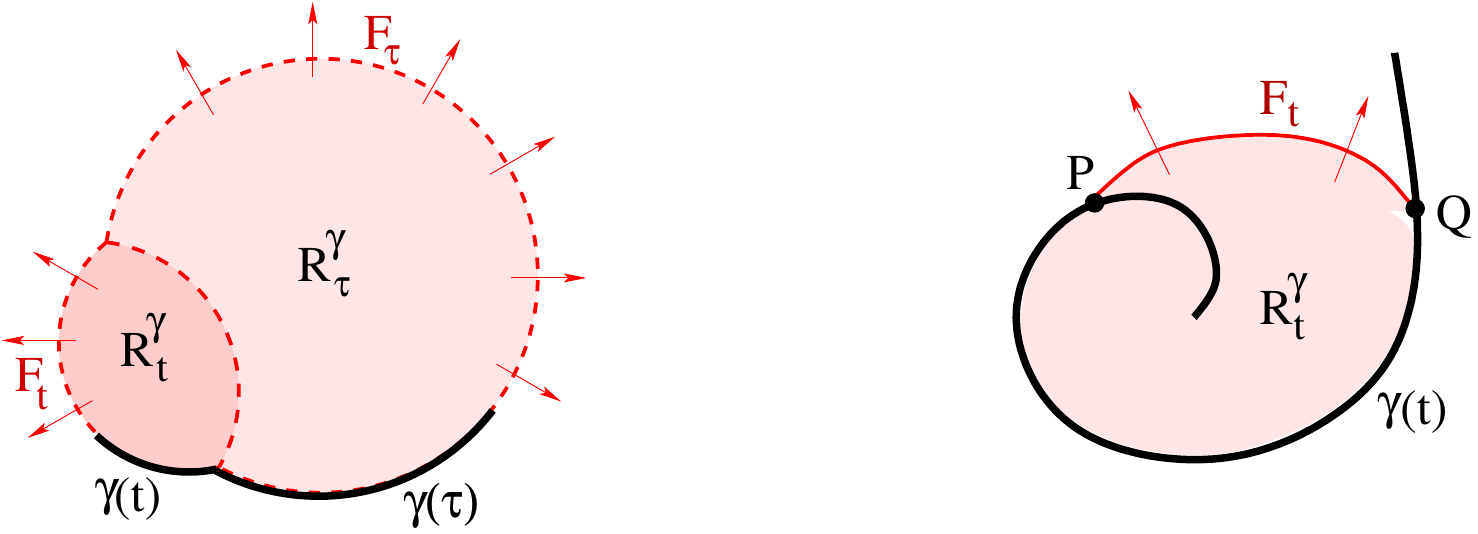}}}
\caption{\small  Left: the solution to the double-bubble problem provides a lower bound on the 
lengths of the portions $F_t$, $F_\tau$ of the boundaries where the fire expands freely, at the two times $t<\tau$. 
Right: the portion of the barrier $\gamma(t)$ between the points $P$ and $Q$  is the {\em active
wall } $W_t$ at time $t$.   The remaining portions do not restrict the instantaneous fire propagation, because
either they have not yet been reached by the fire, or they are already fully surrounded by it.
}
\label{f:fc208}
\end{figure}

In the following, recalling (\ref{WFt}) we define
$$F(t)~\doteq~\H^1(F_t)$$
to be the length of the advancing fire front, at time $t$.

\begin{proposition}\label{prop:wall}
There is a set $N \subset \mathbb R^+$ such that for every $a, b \in \mathbb R^+\setminus N$, there holds
\begin{equation}\label{eq:two-time-upper}
 Q(a,b)~\le~ F(a)+F(b)
 +\lambda\bigl[b+(\sqrt2-1)a\bigr].
\end{equation}
\end{proposition}

{\bf Proof.}
There holds
\begin{align*}
 Q(a,b)-F(a)-F(b)
 &=\H^1(W_a)+\H^1(W_b)
   -(2-\sqrt2)\H^1(W_a\cap W_b)\\
 &=\H^1(W_a\cup W_b)
   +(\sqrt2-1)\H^1(W_a\cap W_b).
\end{align*}
Since the barrier is increasing in time,
\[
 W_a\cup W_b\subseteq\gamma(b),
 \qquad
 W_a\cap W_b\subseteq\gamma(a).
\]
Using \eqref{1.4} gives \eqref{eq:two-time-upper}.
\endproof
%
%This is the first gain over the one-time argument.  If the two perimeters were
%estimated separately, the wall cost would be $\lambda(a+b)$.  Instead, the
%common wall is paid only once, up to the smaller extra factor $\sqrt2-1$.

\subsection{Double-bubble variational problems}
We consider the following  optimization problem for sets $\Omega_1, \Omega_2 \subset \mathbb R^2$ of finite perimeter in the plane. As a convention, in the following let $\Omega_0 := \mathbb R^2 \setminus (\Omega_1 \cup \Omega_2)$, and let 
$$
\Gamma_{i,j} := \partial^\ast \Omega_i\cap \partial^\ast \Omega_j.
$$
Here $\Gamma_{i,j}$ are disjoint up to $\H^1$ negligible sets.

\begi
\item[{\bf (DBP)}] {\bf (Double Bubble Problem).} {\it Given $m,n\geq 0$, find two disjoint regions $\Omega_1,\Omega_2\subset
\R^2$ with areas $m,n$ respectively, which minimize the weighted sum of perimeters:
\bel{JB}
%J(\Omega_1,\Omega_2)~\doteq~
\H^1(\Gamma_{12}) + \H^1(\Gamma_{20})+ \sqrt 2 \cdot\H^1(\Gamma_{10}).\eeq
% Here 
% $$\Gamma_{12}\,\doteq\,\partial^*\Omega_1\cap\partial^*\Omega_2$$
% is the common boundary of the two sets, while 
% $$\Gamma_{10} \,\doteq\, \bd \Omega_1\setminus \Gamma_{12}\,,\qquad \Gamma_{20} \,\doteq\, \bd \Omega_2\setminus \Gamma_{12}\,,$$
% are the remaining portions of the two boundaries.
}\endi
%
%Call $\D(m,n)$  the minimum weighted perimeter (\ref{JB}) of a partition of the plane
%into two bounded chambers of areas $m$ and $n$ and the exterior, with surface
%tensions
%\begin{equation}\label{eq:weights}
% \sigma_{12}=1,
% \qquad
% \sigma_{20}=1,
% \qquad
% \sigma_{10}=\sqrt2.
%\end{equation}
The weighted double-bubble theorem of Lawlor~\cite{Lawlor} identifies the
minimizer with the standard weighted double bubble made of three circular
arcs, as shown in Fig.\,\ref{f:fc209}.   The minimum cost $\D(m,n)$ has the following properties:
\bel{scaling} \D(cm,cn)\,=\,\sqrt c\,\D(m,n), \qquad \forall c>0,
 \eeq
 \bel{endpoints} \D(0,n)\,=\,2\sqrt{\pi n},
 \qquad
 \D(m,0)\,=\,2\sqrt{2\pi m}.
\eeq

\begin{remark}\label{r:41} {\rm 
Applying the definition of $\D$ to the two sets
$\Omega_1=R^\gamma_a$ and $\Omega_2= R^\gamma_b\setminus R^\gamma_a$, one immediately obtains
\bel{eq:bubble-lower}
 Q(a,b)~\ge~
 \D\bigl(A(a),~A(b)-A(a)\bigr).
\end{equation}
}
\end{remark}

%There is an important subtlety here.  Later we shall know only that
%\[
% A(a)\ge ca^2,
% \qquad
% A(b)\ge cb^2.
%\]
%This does \emph{not} imply
%$A(b)-A(a)\ge c(b^2-a^2)$.  Therefore one cannot simply replace the two
%chamber areas in \eqref{eq:bubble-lower}.  The next lemma is what makes the
%argument rigorous.

The following monotonicity property will also be useful.
\begin{lemma}\label{l:mono}
Let $\D(m, n)$ be the minimum cost for the problem {\bf (DBP)} with areas $m, n > 0$. Let $v_0>0$ and $0\le m_0\le v_0/4$. Then for all $0 \leq m \leq v$ there holds
\begin{equation}\label{eq:monotonicity}
     \D(m,v-m)~\ge~ \D(m_0,v_0-m_0) \qquad \forall \; 0 < v_0 \leq v, \quad 0 \leq m_0 \leq \min \{m, \frac{v_0}{4}\}
\end{equation}
% If
% \[
%  m\ge m_0,
%  \qquad
%  v\ge v_0,
%  \qquad
%  0\le m\le v,
% \]
% then
% \begin{equation}\label{eq:monotonicity}
%  \D(m,v-m)~\ge~ \D(m_0,v_0-m_0).
% \end{equation}
\end{lemma}

\noindent{\bf Proof.}
Set
\[
G(m,v)\,\doteq\,\D(m,v-m),\qquad 0\leq m\leq v.
\]
Here $m$ is the area of the first chamber and $v$ is the total enclosed
area. Increasing $v$ with $m$ fixed adds area to the second chamber,
whereas increasing $m$ with $v$ fixed transfers area from the second
chamber to the first.

We shall prove that, for fixed $v$, the minimum of $G$ on $[m_0,v]$
is attained at one of its endpoints, and then compare both
endpoint values with $G(m_0,v_0)$.

\begin{figure}[ht]
\centerline{\hbox{\includegraphics[width=8cm]{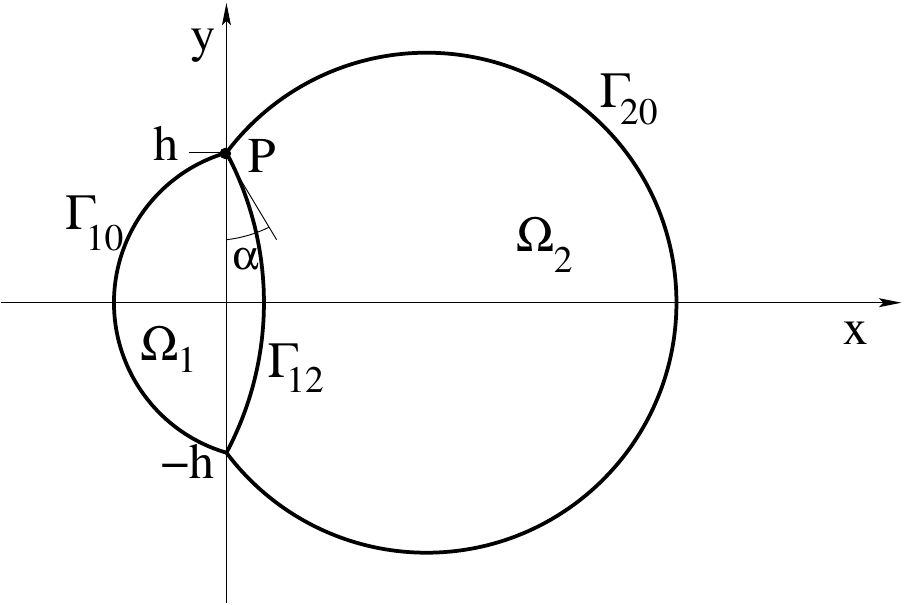}}}
\caption{\small The optimal shape of the double bubble.}
\label{f:fc209}
\end{figure}

We first recall the junction condition for a weighted double bubble;
see \cite[Section~6]{Lawlor}. At each junction, let $T_1,T_2,T_3$
be vectors tangent to $\Gamma_{10},\Gamma_{20},\Gamma_{12}$, respectively,
directed away from the junction, normalized so that
\[
|T_1|=\sqrt2,\qquad |T_2|=|T_3|=1.
\]
The necessary condition for optimality yields
\begin{equation}\label{eq:junction-balance}
T_1+T_2+T_3\,=\,0.
\end{equation}
Taking the scalar product of \eqref{eq:junction-balance} with $T_i$, $i_1, 2, 3$,  one obtains
\[\langle
T_2 , T_3\rangle\,=\,0,
\qquad \langle T_1, T_2\rangle\,=\,\langle T_1, T_3\rangle\,=-1.
\]
Thus the two arcs $\Gamma_{12}$ and $\Gamma_{20}$ having weight 1 meet orthogonally, while the arc 
$\Gamma_{10}$ with 
weight $\sqrt2$ forms an angle $3\pi/4$ with $\Gamma_{21}$ and $\Gamma_{20}$.

\medskip
\noindent{\bf Step 1.}
Place the two junctions at $(0,\pm h)$, with $h>0$.
Let $\alpha$ be the signed angle between the downward vertical
and the tangent to the internal arc at the upper junction,
positive when the arc bends into the second chamber (see Figure \ref{f:fc209}).

At the upper junction, the tangent relations just established give
\begin{equation}\label{eq:upper-tangents}
\begin{aligned}
T_3&=(\sin\alpha,-\cos\alpha),\\
T_2&=(\cos\alpha,\sin\alpha),\\
T_1&=(-\sin\alpha-\cos\alpha,\cos\alpha-\sin\alpha).
\end{aligned}
\end{equation}
Here the direction of $T_2$ is chosen along the exterior arc of
the second chamber, and $T_1$ then follows from
\eqref{eq:junction-balance}.
Since both junctions lie on the vertical axis, the exterior circles
of $\Omega_1$ and $\Omega_2$ have radii
\begin{equation}\label{eq:exterior-radii}
\frac{\sqrt2h}{\sin\alpha+\cos\alpha},
\qquad
\frac{h}{\cos\alpha},
\end{equation}
respectively. Imposing that these radii are positive, we obtain the parameter range 
$$-\pi/4<\alpha<\pi/2.$$

In this step, we prove that the area fraction
\[
q=\frac{|\Omega_1|}{|\Omega_1|+|\Omega_2|}
\]
is a smooth, strictly decreasing function of $\alpha$ on this
interval, with limits $1$ and $0$ at its two endpoints. This is a quite intuitive fact: when $\alpha$ increases, from Figure \ref{f:fc209} we expect the area fraction of the second chamber to increase. This can be justified by elementary calculations, or the following scaling/uniqueness argument. By contradiction, assume that $q(\alpha)$ is not a monotone function of $\alpha$. Then for two values $\alpha_1, \alpha_2$, one has $q(\alpha_1) = q(\alpha_2)$. However, we know that (see e.g. \cite{Lawlor})  that for a given $q$ there exists a unique double bubble satisfying the tangent relations \eqref{eq:junction-balance} with given total area. Since the angles $\alpha_i$ are not changed by dilations, this would contradict the above uniqueness of double bubbles. 

Thus the map $\alpha\mapsto q(\alpha) $ is a smooth bijection from
 $\,]-\pi/4,\pi/2[\,$ onto $\,]0,1[\,$ with nonvanishing derivative.
 The inverse function theorem gives a smooth inverse $q\mapsto \alpha(q)$,
 satisfying
 \[
 \alpha(0+)=\frac{\pi}{2},
 \qquad
 \alpha(1-)=-\frac{\pi}{4}.
 \]

\medskip
\noindent{\bf Step 2.}
We now prove that, for fixed $v$, the function $m\mapsto G(m,v)$
first increases and then decreases. This will imply
\[
G(m,v)\geq\min\{G(m_0,v),G(v,v)\},
\qquad m_0\leq m\leq v.
\]

We work with bubbles of prescribed total area $v$, translated so
that the junctions occur at the points $(0,\pm h)$. Their scale $h$ is allowed
to vary with $m$.

We calculate the derivatives of the function $\D(m, n)$, which can be done via a standard Lagrange-multiplier argument. In fact, by the Lagrange multiplier theorem, there exist $p_1, p_2$ such that, calling 
$$
\mathcal J(\Omega_1, \Omega_2)\,\doteq\, \H^1(\Gamma_{12}) + \H^1(\Gamma_{20})+ \sqrt 2 \cdot\H^1(\Gamma_{10})
$$
there holds, for an optimal configuration $\Omega_1, \Omega_2$ with areas $m, n$, that
$$
\delta \mathcal J(\Omega_1, \Omega_2) ~=~ p_1 \delta|\Omega_1| + p_2 \delta |\Omega_2|.
$$
Here $p_i$ must be the curvature of the boundary of the exterior arc of $\Omega_i$ multiplied by the weight at that point. By \eqref{eq:exterior-radii}, the curvatures of the two exterior
arcs, multiplied by their respective weights, are
\[
\frac{\cos\alpha+\sin\alpha}{h},
\qquad
\frac{\cos\alpha}{h}.
\]
Notice in passing that the curvature coefficient on the internal arc, for a displacement
from $\Omega_1$ towards $\Omega_2$, is
\[
\frac{\sin\alpha}{h}.
\]
In particular, this is the difference of the two weighted
exterior curvatures.
Thus
\begin{equation}\label{eq:bubble-first-variation}
d\D
=
\frac{\cos\alpha+\sin\alpha}{h}\,d|\Omega_1|
+
\frac{\cos\alpha}{h}\,d|\Omega_2|.
\end{equation}
Since $|\Omega_1|=m$ and $|\Omega_2|=v-m$, this becomes
\[
dG
=
\frac{\sin\alpha}{h}\,dm
+
\frac{\cos\alpha}{h}\,dv.
\]
Now Step~1 provides
a smooth family of minimizing configurations for $0<m<v$, parametrized by $\alpha(q) = \alpha(m/n)$.
Keeping $v$ fixed, we thus obtain
\begin{equation}\label{eq:G-m-derivative}
\frac{\partial G}{\partial m}(m,v)
=\frac{\sin\alpha(m/v)}{h}.
\end{equation}
Geometrically, transferring area from the second chamber to
the first increases the cost when the internal arc bends into
the second chamber. The derivative vanishes when the interface
is straight, and becomes negative when it bends into the first
chamber. Its sign therefore depends only on $\alpha$.

By Step~1, $\alpha(m/v)$ decreases strictly from $\pi/2$
to $-\pi/4$ as $m$ increases from $0$ to $v$. Since $h>0$, the right hand side of
\eqref{eq:G-m-derivative} is positive until the unique configuration
with $\alpha=0$, and becomes negative afterwards. Hence
$m\mapsto G(m,v)$ has exactly one interior maximum and no
interior minimum.

The function $G(m, v)$ can be easily calculated at the endpoints; in fact, when the total area is equal to $v$, the optimal configuration $G(0, v)$ is just a circle, with perimeter weighted by $1$, while $G(v, v)$ is again a circle, but with perimeter weighted by $\sqrt{2}$. This yields immediately
\[
G(0,v)=2\sqrt{\pi v},
\qquad
G(v,v)=2\sqrt{2\pi v}.
\]
Since $m \mapsto G(m ,v)$ has a unique local maximum in $[m_0, v]$, we obtain
\[
G(m,v)~\geq~\min\{G(m_0,v),G(v,v)\},
\qquad m_0\leq m\leq v,
\]
as claimed.

\medskip
\noindent{\bf Step 3.}
In this step we show that 
\[
G(m_0,v_0)
\,<\,G(v_0,v_0).
\]

Consider the bubble with $\alpha=\pi/4$ and $h=1$.
By the tangency relations \eqref{eq:upper-tangents}, the circles
carrying $\Gamma_{10}$, $\Gamma_{20}$ and $\Gamma_{12}$ have
centers $(0,0)$, $(1,0)$ and $(-1,0)$, and radii
$1$, $\sqrt2$ and $\sqrt2$, respectively.
% The first exterior arc is a semicircle, the second exterior arc
% is three quarters of a circle, and the internal arc is a quarter
% of a circle.

% The internal arc and the segment joining the two junctions enclose
% a circular segment of area $\pi/2-1$: this is a quarter-disk
% of radius $\sqrt2$ minus a triangle of area $1$.
% The region between the second exterior arc and the same segment
% has area $2\pi-(\pi/2-1)=3\pi/2+1$.
% Adding the internal circular segment to the first semicircle,
% and subtracting it from the latter region, gives
It is easy to see that in this case
\[
\begin{aligned}
|\Omega_1|&=\pi-1,\\
|\Omega_2|&=\pi+2.
\end{aligned}
\]
The lengths of the two exterior arcs and the internal arc are,
respectively,
\[
\Gamma_{1,0} = \pi,\qquad \Gamma_{20} = \frac{3\pi}{\sqrt2},
\qquad \Gamma_{12} = \frac{\pi}{\sqrt2}.
\]
Thus the total area is $v = 2\pi+1$ and the weighted perimeter is
\[
\D(\pi-1, \pi +2) = \sqrt2\,\pi+\frac{3\pi}{\sqrt2}
                 +\frac{\pi}{\sqrt2}=3\sqrt2\,\pi.
\]
Dilating this bubble to total area $v_0$, using \eqref{scaling}, gives
\[
G\left(\frac{\pi-1}{2\pi+1}v_0,v_0\right)
=3\sqrt2\,\pi\sqrt{\frac{v_0}{2\pi+1}}.
\]
Since
\[
\frac{m_0}{v_0}\leq\frac14
<\frac{\pi-1}{2\pi+1},
\]
by Step~1 we obtain that that $\alpha\geq\pi/4$ whenever
$$
0<m\leq \frac{(\pi-1)v_0}{(2\pi+1)} =: m^\star.
$$
By Step~2 and continuity at $m=0$, the map $m \mapsto G(m,v_0)$ is
increasing on $[0,m^\star]$.
Therefore we finally get
\[
G(m_0,v_0)
~\leq~ G\left(\frac{\pi-1}{2\pi+1}v_0,v_0\right)
~=~3\sqrt2\,\pi\sqrt{\frac{v_0}{2\pi+1}}
~<~2\sqrt{2\pi v_0}
~=~G(v_0,v_0).
\]

\medskip
\noindent{\bf Step 4.}
We now compare both endpoints from Step~2 with $G(m_0,v_0)$.
For $m_0>0$, the first-variation formula in Step~2 yields
\[
\frac{\partial G}{\partial v}(m_0,v)
~=~\frac{\cos\alpha(m_0/v)}{h}~>~0
\qquad \forall v\geq v_0,
\]
because $-\pi/4<\alpha<\pi/2$ and $h>0$.
Thus adding area to the second chamber while keeping the first
one fixed increases the minimum cost.
For $m_0=0$, the same monotonicity follows directly from
$G(0,v)=2\sqrt{\pi v}$.
In either case,
\[
G(m_0,v)\geq G(m_0,v_0).
\]
Concerning the other endpoint, Step~3 and $v\geq v_0$ give
\[
G(v,v)~=~2\sqrt{2\pi v}
~\geq~ 2\sqrt{2\pi v_0}
~>~G(m_0,v_0).
\]
Combining these inequalities with the endpoint-minimum property
established in Step~2, we obtain
\[
G(m,v)
~\geq~\min\bigl\{G(m_0,v),G(v,v)\bigr\}
~\geq~ G(m_0,v_0),
\]
which is the desired conclusion.
\endproof

\subsection{An improved lower bound}
Define
\begin{equation}\label{Kdef}
\begin{aligned}
 K&\,\doteq\,\int_0^{1/2}
       \D\!\left(s^2,s+\frac14\right)\,ds,\\
 \mu&\,\doteq\,\int_0^{1/2}
       \left(\frac12+\sqrt2\,s\right)\,ds
       =\frac{2+\sqrt2}{8},
\end{aligned}
\end{equation}
and set
\begin{equation}\label{Lambda}
 c_*\,\doteq\,\frac{K^2}{4},
 \qquad
 \Lambda\,\doteq\,\frac{K^2}{4\mu}
       =(2-\sqrt2)K^2.
\end{equation}
The two arguments of $\D$ correspond to the area at the earlier
time and the area gained between the two times, since
$(s+1/2)^2-s^2=s+1/4$.
A numerical evaluation of the integral defining $K$ gives that the constant $\Lambda$ yields
$$
\Lambda \approx 1.62621170228\ldots
$$

\begin{theorem}\label{thm:main}
Let $\gamma(\cdot)$ be a $\lambda$-admissible strategy and let
$A(t)=\caL^2(R^\gamma_t)$, with $A(0)>0$.
If $\lambda\leq\Lambda$, then
\begin{equation}\label{compar}
 A(t)>c_*t^2\qquad\hbox{for every }t>0.
\end{equation}
In particular, the fire cannot be blocked.
\end{theorem}

\noindent{\bf Proof.}
Suppose that \eqref{compar} fails and let
\[
 T=\inf\{t>0:\ A(t)\leq c_*t^2\}.
\]
Since $A$ is nondecreasing and $A(0)>0$, we have $T>0$ and
\begin{equation}\label{eq:contact}
 A(T)=c_*T^2,
 \qquad A(t)>c_*t^2\quad\hbox{for }0<t<T.
\end{equation}

For almost every $t\in(0,T/2)$, both $t$ and $t+T/2$ are good
times for Proposition \ref{prop:arearate}. Since
\[
 \frac{c_*t^2}{c_*(t+T/2)^2}\leq\frac14,
\]
we can apply Lemma~\ref{l:mono} with $m = A(t)$, $v = A(T)$. Combining the lemma with
\eqref{eq:bubble-lower}, Proposition~\ref{prop:wall}, and
the scaling identity \eqref{scaling}, we obtain
 \bel{eq:pair}\bega{rl} \ds
 \sqrt{c_*}\,T\,
 \D\!\left((t/T)^2,\,t/T+\frac14\right)
 &\leq~\D\Big(A(t),\,A(t+T/2)-A(t)\Big)\notag\\[4mm]
 &\leq \ds~F(t)+F(t+T/2)
       +\lambda\left(\frac T2+\sqrt2\,t\right).
\enda\eeq
Integrating over $(0,T/2)$ and using \eqref{eq:areaineq} to estimate the right hand side gives
\begin{align}
 \sqrt{c_*}\,KT^2
 &\leq\int_0^T F(t)\,dt+\lambda\mu T^2\notag\\[1mm]
 &\leq A(T)-A(0)+\lambda\mu T^2\notag\\[1mm]
 &= (c_*+\lambda\mu)T^2-A(0).
 \label{eq:key}
\end{align}
But $\sqrt{c_*}\,K=2c_*$ and
$\lambda\mu\leq\Lambda\mu=c_*$, hence \eqref{eq:key} implies
\[
 2c_*T^2~\leq ~2c_*T^2-A(0),
\]
contradicting $A(0)>0$.
\endproof

The choice of $c_*$ is optimal within this quadratic comparison
argument. Indeed, with a general coefficient $c>0$, the same
proof gives a contradiction whenever
\[
 \lambda\leq\frac{\sqrt c\,K-c}{\mu}
 =\Lambda-\frac{(\sqrt c-K/2)^2}{\mu}.
\]
The largest threshold is therefore $\Lambda$, attained at $c=c_*$.

\section{Barriers which are Jordan curves}
\label{sec:5}
\setcounter{equation}{0}

In this final section we examine the case where the entire barrier $\Gamma=\cup_{t>0} \gamma(t)$ 
is a Jordan curve.  When $\lambda>2$ it is well known \cite{B} that there exists an admissible strategy
which constructs  two
arcs of logarithmic spirals and completely surrounds the fire. Once completed, the entire barrier is a Jordan curve.

On the other hand, the following intuitive argument (already sketched in \cite{BrP}) shows that a strategy which block the fire by  
constructing  a Jordan curve
can be admissible only
if the construction speed is $\lambda>2$.

Assume that the fire starts from the unit ball centered at the origin.  As shown in Fig.~\ref{f:fc211}, let
 $\bar x\in \Gamma$ be the location where the last brick of the wall is put in place, say at time 
 $T$.   
Then the entire wall  must be constructed before the fire can reach the point $\bar x$.  Since $\lambda $ is the construction rate, this implies $$\lambda T
\,\geq\, \H^1(\Gamma).$$
Let $\gamma$ be a path with minimum length, contained inside the region bounded by $\Gamma$, joining the 
origin with $\bar x$.   By minimality, the first portion of $\gamma$ must be a segment.   We prolong this segment 
backward until it crosses a first point $x_0\in\Gamma$.   Call $\gamma_0$ this longer path connecting
$x_0$ with $\bar x$.   Moreover, call $\Gamma_1, \Gamma_2$ respectively the two portions of $\Gamma$,
covered when traveling
from $x_0$ to $\bar x$ clockwise and counterclockwise, respectively.
Since $\Gamma$ is a simple closed curve, the lengths of these paths satisfy
$$\H^1(\gamma)+1 ~\leq~\H^1(\gamma_0)~\leq~\min\bigl\{ \H^1(\Gamma_1),\,\H^1(\Gamma_2)\bigr\}.$$
Since $T\leq \H^1(\gamma)$ (otherwise the fire escapes from $\bar x$ before the wall is completed),
we thus conclude
$$\lambda T~=~\H^1(\Gamma)~\geq~2\bigl(\H^1(\gamma) +1\bigr)~\geq~2(T+1).$$
Hence $\lambda>2$.

We observe, however, that no contradiction is reached if $\lambda>1$ and the barrier $\Gamma$
is not a simple closed curve. 

\begin{figure}[ht]
\centerline{\hbox{\includegraphics[width=10cm]{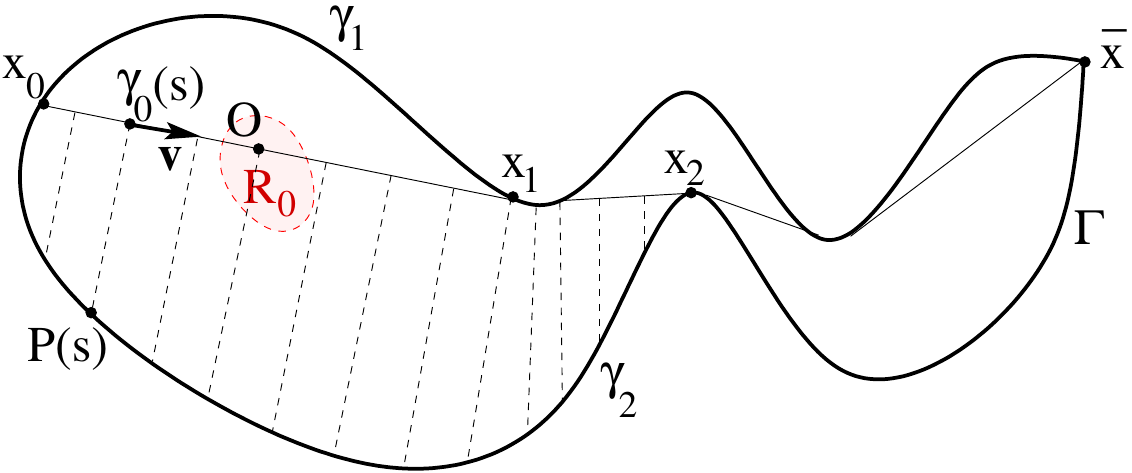}}}
\caption{\small If the entire barrier $\Gamma$ is a simple closed curve, 
then $\Gamma=\Gamma_1\cup\Gamma_2$, where each of the two portions has length greater
than the length of $\gamma_0$.  Therefore, the time $T$ needed for the fire to reach any point of $\Gamma$
is strictly smaller than half of the length of $\Gamma$. }
\label{f:fc211}
\end{figure}

 \v
Toward a more rigorous proof, we recall some results from \cite{BW1}.
Let $\Gamma\subset\R^2$ be a complete rectifiable set.
The set reached by the fire  at time $t$ in the presence of the barrier
$\Gamma$ is then defined as
\bel{61} \begin{array}{rl}R^\Gamma_t~\doteq&\Big\{ x(t)\,;~~x(\cdot)~~
\hbox{is absolutely continuous}\,,~~
x(0)\in R_0\,,\cr
&\qquad\quad |\dot x(\tau)|\leq 1 ~~\hbox{for a.e.}~
\tau\in [0,t]\,,\qquad x(\tau)\notin \Gamma~~
\hbox{for all}~\tau\in [0,t]\Big\}.
\end{array}\eeq
We also define $R^\Gamma_\infty\doteq \bigcup_{t
\geq 0}
R^\Gamma_t$.
We say that the complete, rectifiable set
$\Gamma$ is {\em $\lambda$-admissible} if, for every $t\geq 0$,
the set
\bel{6.2}\gamma(t)~\doteq ~\Gamma\cap \overline{R^\Gamma(t)}\eeq
satisfies (\ref{1.4}), i.e.~it
can be constructed within time $t$.
We now consider:
\v
{\bf (BP$^\prime$)} {\it Find an admissible rectifiable set
$\Gamma\subset\R^2$ such that the corresponding reachable set $R_\infty^\Gamma$
is bounded.}

The main result in \cite{BW1} shows that, under the assumptions {\bf (A1)--(A3)},
the two formulations of the blocking problem
{\bf (BP)} and {\bf (BP$^\prime$)} are equivalent.

Given a
rectifiable set $\Gamma$, we define the corresponding multifunction
$t\mapsto
\gamma(t)$ according to (\ref{6.2}).
Conversely, given an admissible strategy $\gamma(\cdot)$,
we consider the set
\bel{Geq}\Gamma~\doteq ~\hbox{completion of}~~(\gamma_\infty\setminus
R^\gamma_\infty)\,,\qquad \gamma_\infty\doteq  \bigcup_{t>0} \gamma(t),~~R^\gamma_\infty\doteq   \bigcup_{t>0} R^\gamma_t\,. \eeq
The following theorem was proved in \cite{BW1}.

\begin{theorem} In the above setting,
there exists an admissible strategy $t\mapsto \gamma(t)$ with ${\H^1}(\gamma_\infty)<\infty$
which solves the blocking problem {\bf (BP)} if and only if
there exists an admissible rectifiable set $\Gamma$ which solves {\bf (BP$^\prime$)}.
\end{theorem}

We can now prove:
\begin{theorem}\label{t:52}
Let $\Gamma$ be a Jordan curve with finite length, blocking a fire initially starting on a bounded open set $R_0$.
Then $\Gamma$ can be $\lambda$-admissible only if $\lambda>2$.\end{theorem}

{\bf Proof.} 
{\bf 1.} 
Let $\ell=\H^1(\Gamma)$.   We will show that every point $x\in\Gamma$ can be reached by the fire
within time $T \doteq (\ell-r)/2$. This  implies that 
$\ov{R^\Gamma(T)}\,\supset\, \Gamma.$
Hence
$$\H^1 \bigl(  \ov{R^\Gamma(T)}\cap\Gamma\bigr) ~=~\H^1(\Gamma)~=~\ell~>~2T.$$
As a consequence, if  $\lambda\leq 2$, then  $\Gamma$ is not $\lambda$-admissible 
\v
{\bf 2.} 
Denote by $E$ the bounded, open, connected component of $\mathbb R^2 \setminus \Gamma$.
Note that its boundary is $\partial E=\Gamma$, 
while its closure is $\ov E=E\cup \Gamma$.
By a translation, we can assume that the initial burned set $R_0\subset E$ contains the open ball 
$B(0,r)$ centered at the origin with radius $r$. 
We consider the path distance between points in $\ov E$,
$$\bega{rl} 
d_E(x_1, x_2) &\ds\doteq~\min\bigg\{ \int_0^1 \bigl|\dot x(t)\bigr|\,d t\,;\quad x:[0,1]\mapsto \ov E~
\hbox{ is an absolutely continuous path}\\[3mm]
&\qquad\qquad\qquad  \hbox{with}~~
x(0) = x_1, ~~ x(1) = x_2\,\bigg\}.\enda
$$
Since $\Gamma$ has bounded length, the existence of a minimizing path follows from a standard compactness argument.

As proved in Lemma~2 of \cite{BW2}, since the Jordan domain $E$ is simply connected,
the maximum distance between two points in $\ov E$ is achieved on the boundary:
\begin{equation}\label{eq:boundarydist}
\sup_{x_1,x_2 \in \ov E} d_E(x_1, x_2)  ~=~ \sup_{x_1,x_2 \in \Gamma} d_E(x_1, x_2).
\end{equation}
\v
{\bf 3.} 
Given any two distinct points  $x_1,x_2\in\Gamma$, let
$\Gamma_1,\Gamma_2\subset\Gamma$
be the arcs travelled by moving along $\Gamma$ from $x_1$ to $x_2$ in the clockwise and in the counterclockwise direction, respectively. Clearly $ \ell=\H^1(\Gamma) =\H^1(\Gamma_1) + \H^1(\Gamma_2)$,
and hence
\bel{perest}
d_E(x_1,x_2)~\leq~\min \bigl\{  \H^1(\Gamma_1) + \H^1(\Gamma_2)\bigr\} ~\leq~{\ell\over 2}\,.\eeq

{\bf 4.}  Next, consider any point $\bar x \in \Gamma$. 
%We claim that it is reached by the fire starting from $R_0$  within time 
%$$T~=~{\ell-r\over 2} \,.$$
%
By  \eqref{perest} there is a path $\gamma^{\bar x}:[0,1]\mapsto \ov E$ starting from the origin and reaching $\bar x$,  with length $\le \ell/2$.
%$$
%d_E(0, \bar x) \, \leq\, \frac{\ell}{2}.
%$$
Since $R_0$ contains a ball of radius $r$ centered at $0$, the path $\gamma^{\bar x}$ contains an initial
arc  connecting the origin to a point $x^*\in R_0$ with $|x^*|=r/2$.  Removing this arc, we obtain a shorter path
$\gamma^*$ connecting $x^*$ with $\bar x$, with length $\leq (\ell-r)/2$.
Hence
$$d_{E}(x^*, \bar x) ~\leq ~\frac{\ell-r}{2}\,.
$$
\v
{\bf 5.}
 In this final step we prove that, for any $\ve>0$, there exists a
path $\gamma_\ve$ starting inside $R_0$ and reaching a point $\bar x_\ve\in B(\bar x, \ve)$
within time $\leq (\ell-r)/2$, without ever touching the barrier $\Gamma$.

For sufficiently small $\delta>0$, call $B(\Gamma, \delta)$ the open neighborhood with radius $\delta>0$
around the set $\Gamma$ and  let $E_\delta$ be the connected component 
of $\R^2\setminus B(\Gamma, \delta)$ which intersects $R_0$.  Note that $E_\delta$ is simply connected.
 By choosing $\delta>0$ 
small enough, we also have 
\bel{Omde}
B(x^*, r/2)\subset E_\delta\,,\qquad\qquad E_\delta\cap B(\bar x, \ve)\not= \emptyset.\eeq

We now repeat the previous argument, replacing $\Gamma$ with the boundary $\partial E_\delta$
and $\bar x$ with a point $\bar x_\ve\in \partial E_\delta\cap B(\bar x,\ve)$.
This yields the existence of a path $\gamma_\ve$ , entirely contained inside $E_\delta$,
joining the point  $x^*\in R_0$ with $\bar x_\ve$.   Moreover, the length of this path satisfies
$$\H^1(\gamma_\ve) + {r\over 2} ~\leq~{1\over 2} \H^1(\partial E_\delta).$$

Since $\Gamma$ is a Jordan curve with finite length, we have \cite{ACV}
$$\liminf_{\delta\to 0+} ~\H^1(\partial E_\delta)~=~\H^1(\Gamma).$$
We thus conclude that 
$$\Gamma\subset \ov{R^\Gamma_T}\qquad\hbox{for} ~~T\,=\,{1\over 2} \left(\H^1(\Gamma) - {r\over 2}\right).
$$
Hence $\Gamma$ is not $\lambda$-admissible for any $\lambda\leq 2$.
\endproof
\v

We observe that the condition that $\Gamma= \cup_{t>0}\gamma(t)$ is a Jordan curve implies
\bel{bub}\gamma(t)\cap \ov {R^\gamma_t}~\subseteq~\partial^u(\R^2\setminus \ov{R^\gamma_t}) \qquad\qquad\forall t>0,\eeq
where $\partial^u \Omega$ denotes the boundary of the unbounded component of a set $\Omega$.

When the construction speed is $\lambda<2$, 
the next theorem gives a lower bound on the area of the burned set, assuming that (\ref{bub}) holds.
As before, by $|\Omega|$ we denote the Lebesgue measure of a set $\Omega\subset\R^2$.

\v
\begin{theorem}
For any $\lambda<2$, setting $c_\lambda \doteq (2-\lambda)/4$,  the following holds.
If $\gamma(\cdot)$ is a $\lambda$-admissible strategy 
with the property (\ref{bub}), then  the set burned by the fire satisfies
\bel{bibu} \bigl|R^\gamma_t\bigr|~\geq~ c_\lambda  t^2\qquad\qquad\forall t\geq 0.\eeq
\end{theorem}

{\bf Proof.} {\bf 1.} By contradiction, assume that there exists a $\tau > 0$  such that
$$
\bigl|R^\gamma_\tau\bigr|~ \leq ~c_\lambda \tau^2,
$$
where $\gamma$ is a $\lambda$-admissible strategy with the property (\ref{bub}). 
We claim that in this case we can select an earlier time $\tau' \in [\tau/2, \tau]$ such that 
$$\H^1(F_{\tau'})~ \leq~ 2 c_\lambda \tau\,.$$
Otherwise, if $\H^1(F_t)> 2 c_\lambda \tau$ for all $t\in [\tau/2,\,\tau]$, then  by (\ref{eq:areaineq})
$$\bigl|R^\gamma_\tau\bigr| ~\geq ~|R_0| + \int_{\tau/2}^\tau  \H^1(F_t) \, dt~>~  \bigl|R_0\bigr| +c_\lambda \tau^2, 
$$
reaching a contradiction.
\v
{\bf 2.}
Since $\tau' \geq  \tau/2$, we deduce 
$$
\H^1(F_{\tau'}) ~\leq ~2 c_{\lambda} \tau  ~\leq~ 4 c_\lambda \tau'.
$$
Now take
$$
\lambda^\prime ~\doteq~\lambda + 4 c_\lambda.
$$
and set 
$$\Gamma\,\doteq\, \partial^u(\R^2\setminus R^\gamma_{\tau'}).$$
Since $R^\gamma_{\tau'} $ is connected, $\Gamma$ is a Jordan curve with length
\bel{H1G}\H^1(\Gamma)~\leq~ \H^1\bigl( \partial R_{\tau^\prime}^\gamma\bigr)~  \leq~ \lambda \tau'  + 4 c_\lambda \tau'~ =~ \lambda^\prime \tau'.\eeq
We claim that the strategy
\bel{Tgam}\Tilde \gamma(t)~\doteq~\left\{ \bega{cl} \gamma(t)\cap \ov{R^\gamma_t}\qquad &\hbox{if}\quad 0\leq t < \tau',\\[2mm]
\Gamma 
 \qquad &\hbox{if}\quad t \geq  \tau',\enda\right.\eeq
 is $\lambda'$-admissible.
 Indeed, for $0<s< t< \tau'$ we have
 $$R^{\Tilde\gamma_t}~=~R^\gamma_t\,,\qquad\qquad \Tilde\gamma(s)\subseteq \Tilde\gamma(t).$$
 Moreover, when $s<\tau'$  by (\ref{bub}) one has
 $$\Tilde\gamma(s) ~=~ \gamma(s)\cap \ov{R^\gamma_s}~ \subseteq ~ \gamma(\tau')\cap \ov{R^\gamma_{\tau'}}
 ~\subseteq~\partial^u\bigl(\R^2\setminus \ov{R^\gamma_{\tau'}}\bigr) ~=~\Gamma.$$
 This shows that the sets $\Tilde\gamma(t)$ satisfy {\bf (A1)}.
 To prove {\bf (A2)} we observe that
 $$\H^1\bigl(\Tilde \gamma(t)\bigr)~\leq~\H^1\bigl( \gamma(t)\bigr)~=~\lambda t~<~\lambda't,$$
 while for $t\geq\tau'$ this admissibility condition follows from (\ref{H1G}).

\v
{\bf 3.}
Adopting the $\lambda'$-admissible strategy (\ref{Tgam}), the fire propagation is completely stopped at time $\tau'$.
However, by Theorem~\ref{t:52} this is possible only if $\lambda' = \lambda + 4 c_\lambda >2$.
Taking $c_\lambda= (2-\lambda)/4$ we thus obtain a contradiction.
This proves the result.
\endproof

\v
{\bf Use of AI tools.} The authors acknowledge the use of ChatGPT as an auxiliary research tool. While the idea of considering two distinct times to control the evolution of the burned region arises naturally, the specific idea of using a weighted double-bubble configuration as a comparison was suggested by GPT-6 Astra, together with the weaker estimate of Section 3. The manuscript was written entirely by the authors, who take full responsibility for its mathematical content and correctness.

{\bf Acknowledgements.} The second author is partially supported by the GNAMPA--INDAM Project 2026, 'Problemi di controllo e regolarità per evoluzione di insiemi'.

\end{document}